\documentclass[11pt,reqno]{amsproc}
\usepackage{geometry}
\usepackage{amsmath, amsthm, amssymb}
\usepackage[colorlinks=true, pdfstartview=FitV, linkcolor=blue,citecolor=blue, urlcolor=blue]{hyperref}
\usepackage[abbrev,lite,nobysame]{amsrefs}
\usepackage{times}
\usepackage[usenames,dvipsnames]{color}
\usepackage{bbm}
\usepackage{mathtools}
\usepackage{subcaption}
\definecolor{labelkey}{rgb}{0,0,1}
\definecolor{purple}{rgb}{0.4, 0.22, 0.33}

\usepackage{dsfont}

\def\eps{\varepsilon}

\def\dd{{\rm d}}
\def\ddt{{\frac{\dd}{\dd t}}}
\def\R {\mathbb{R}}

\def\de{{\partial}}
\newcommand\uu {\boldsymbol{u}}

\def\gr {\boldsymbol{g}}

\def\x{\textbf{x}}
\definecolor{orange}{rgb}{1.0, 0.49, 0.0}
\definecolor{green}{rgb}{0.0, 0.5, 0.0}
\definecolor{brown}{rgb}{0.43, 0.21, 0.1}

\newtheorem{proposition}{Proposition}[section]
\newtheorem{theorem}[proposition]{Theorem}

\newtheorem{lemma}[proposition]{Lemma}
\theoremstyle{definition}

\newtheorem{remark}[proposition]{Remark}

\numberwithin{equation}{section}

\title[Decay rates for the partially dissipative Boussinesq equations]{Asymptotic behavior of 2D stably stratified fluids\\ with a damping term in the velocity equation}

\author[R. Bianchini]{Roberta Bianchini}
\address{IAC, Consiglio Nazionale delle Ricerche, I-00185 Rome, Italy}
\email{roberta.bianchini@cnr.it}

\author[R. Natalini]{Roberto Natalini}
\address{IAC, Consiglio Nazionale delle Ricerche, I-00185 Rome, Italy}
\email{roberto.natalini@cnr.it}

\begin{document}
	\maketitle
	
\begin{abstract}
This article deals with the asymptotic behavior of the two-dimensional inviscid Boussinesq equations with a damping term in the velocity equation.
Precisely, we provide the time-decay rates of the smooth solutions to that system. The key ingredient is a careful analysis of the Green kernel of the linearized problem in Fourier space, combined with bilinear estimates and interpolation inequalities for handling the nonlinearity.
\end{abstract}
	
\section{Introduction}
An important class of variable density fluids is represented by the non-homogeneous incompressible Euler equations in $\mathbb{R}^2$:
\begin{equation}\label{eq:nonlin}
	\begin{aligned}
	 \rho(\de_t+\uu\cdot\nabla)\uu+\nabla  p&=-\rho \, \gr,\\
	(\de_t+\uu\cdot\nabla)\rho&=0,\\
	\nabla\cdot \uu&=0, 
	\end{aligned}
\end{equation}
where $\uu=(u, w) $ is the velocity field, $\rho$ is the density, $p$ is the pressure and  $\gr=(0,g)$ is gravity and all the functions depend on $t\geq 0$ and $(x,y)\in \mathbb{R}^2$. Equations \eqref{eq:nonlin} find a wide application in oceanography, see for instance \cite{Rieutord}, where both the incompressibility and inviscid assumptions are very good approximations of the reality. Taking \eqref{eq:nonlin} as a starting point, it is customary to introduce some additional hypotheses for obtaining the so-called Boussinesq equations, which are formally introduced as follows.
In many physical systems of non-homogeneous fluids, the variations of the density profile are negligible compared to its (constant) average. One then assumes that the equilibrium stratification is a \emph{stable} profile $\bar \rho (x, y) = \bar \rho(y)$, with $\de_y\bar \rho (y) <0$. 
Among all the possible stratification's equilibria, one usually takes into account
locally affine profiles, so that $\de_y\bar{\rho}(y)$ is constant, see \cite{BDSR, Lannes-Saut, MCZD20} and references therein.
We linearize equations \eqref{eq:nonlin} around the \emph{hydrostatic equilibrium}, namely a steady solution with zero velocity field such that
\begin{align*}
(\rho, u, w, p)=(\bar \rho(y), 0, 0,  \bar p(y)),
\end{align*}

where $\bar p'(y)=-g \bar \rho$. More precisely, we consider the following expansions.
\begin{align*}
\rho(t,x,y)&=\bar \rho(y)+ \tilde{\rho}(t,x,y); \notag\\
u(t,x,y)&=\tilde{u}(t,x,y), \; w(t,x,y)=\tilde{w}(t,x,y); \notag\\
p(t,x,y)&=  \bar p(y)+\rho_0\tilde{P}(t,x,y),
\end{align*}

with $\bar \rho (y)=\rho_0+r(y)$, where $\rho_0$ is the (constant) averaged density and $r(y)$ is a function of the vertical coordinate such that $r'(y) < 0$. Thus we plug the previous expansions in system \eqref{eq:nonlin} and we further apply the \emph{Boussinesq approximation}, see \cite{Rieutord}, which consists in neglecting the density variations everywhere but in the gravity terms. More precisely, let us focus on the equation for the vertical velocity $w$ in system \eqref{eq:nonlin} and plug there the above expansions. We obtain
\begin{align*}
(\bar \rho (y) + \tilde \rho) ( \partial_t \tilde w + \tilde{\textbf{u}} \cdot \nabla \tilde w ) - g \bar \rho (y) + \rho_0 \partial_y \tilde P = -g \bar \rho (y) - g \tilde \rho,
\end{align*}
which yields
\begin{align*}
\partial_t \tilde w + \frac{\rho_0 \partial_y \tilde P}{(\bar \rho (y) + \tilde \rho)} = - g \frac{\tilde \rho}{(\bar \rho (y)+\tilde \rho)} - \tilde{\textbf{u}} \cdot \nabla \tilde w.
\end{align*}
In accordance with the Boussinesq approximation, we neglect the density fluctuations and replace $\bar \rho (y) + \tilde \rho$ with $\rho_0$, so that the equation reads
\begin{align*}
\partial_t \tilde w + \partial_y \tilde P = - \frac{g}{\rho_0} \tilde \rho - \tilde{\textbf{u}} \cdot \nabla \tilde w.
\end{align*}

We write the complete system below.
Hereafter we drop the $\emph{tilde}$ for lightening the notation. We obtain the following system

\begin{equation*}
\begin{cases}
\partial_t \rho + \bar \rho'(y) w&=-\textbf{u} \cdot \nabla \rho, \\
\partial_t u + \partial_x P &=- \textbf{u} \cdot \nabla u, \\
\partial_t w + \partial_y P & = - \frac{g}{\rho_0} \rho - \textbf{u} \cdot \nabla w, \\
\partial_x u + \partial_y w & = 0.
\end{cases}
\end{equation*}

Now define $b=\dfrac{g}{\rho_0} \rho$. The equations read

\begin{equation}\label{eq:system-vel-nodamp}
\begin{cases}
\partial_t b - N^2 w &=-\textbf{u} \cdot \nabla b, \\
\partial_t u + \partial_x P &=- \textbf{u} \cdot \nabla u, \\
\partial_t w + \partial_y P & =- b  - \textbf{u} \cdot \nabla w, \\
\partial_x u + \partial_y w & = 0,
\end{cases}
\end{equation}

where the Brunt-V\"ais\"al\"a frequency is given by
\begin{align}
N^2=-\dfrac{g \bar \rho'(y)}{\rho_0}
\end{align}

and $\bar \rho'(y)<0$ since the stratification is stable. We refer to \cite{Rieutord} for a more physically detailed derivation of system \eqref{eq:system-vel-nodamp}, where vertical variations of the pressure are balanced by gravity ($\partial_y \bar{p}=-g \bar{\rho}$). In other words, in the Boussinesq regime the restoring force of equilibrium's fluctuations is gravity (\emph{Archimedes' principle}).

In this article, we investigate the two-dimensional Boussinesq equations with a damping term in the velocity equation, whose interest in applications in electrocapillarity is discussed for instance in \cite{CastroCordoba2019} and references therein, while from the mathematical viewpoint the damping term can be seen as a limit case of  fractional diffusion. Choosing a constant parameter $\alpha >0$, the system reads as follows:

\begin{equation}\label{eq:system-vel}
\begin{cases}
\partial_t b - N^2 w &=-\textbf{u} \cdot \nabla b, \\
\partial_t u + \partial_x P &= - \alpha u - \textbf{u} \cdot \nabla u, \\
\partial_t w + \partial_y P & = - b - \alpha w  - \textbf{u} \cdot \nabla w, \\
\partial_x u + \partial_y w & = 0,
\end{cases}
\end{equation}

We rewrite the equations in vorticity-stream formulation, by introducing the unknown variables

\begin{align}
\omega:=\nabla^\perp \cdot \textbf{u}=(\partial_y, \, -\partial_x) \cdot (u, w), \quad
\textbf{u}=(u, w)=:\nabla^\perp \Phi=(\partial_y \Phi, \, -\partial_x\Phi).
\end{align}

In terms of the new variables, one obtains:
\begin{equation}\label{eq:vorticity-system_full}
\begin{cases}
\partial_t b + N^2 \Delta^{-1} \partial_x \omega & =-\uu \cdot \nabla b, \\
\partial_t \omega + \alpha \omega - \partial_x b &= - \uu \cdot \nabla \omega, \\
\Delta \phi=\omega&=\nabla^\perp \uu.
\end{cases}
\end{equation}

Geophysical fluids gained the interest of the mathematical community a long time ago, we refer to \cite{CDGG2006} for an introduction. 
Well-posedness and stability results for stratified fluids are provided for instance in \cite{Charve17, BDSR, Danchin Paicu2011, DWZZ2018} and references therein. We now mention some previous works concerning the 2D inviscid Boussinesq system \eqref{eq:vorticity-system_full} (with or without damping) in the context of smooth solutions. 
In \cite{Elgindi2015}, the authors obtain almost global existence of solutions $(b, \omega)$ to the inviscid system in vorticity-stream formulation without any damping term, with  $\omega \in H^s \cap H^{-1}, \,  b \in H^{s+1}$,  $s>4$, and initial data such that $ \omega_0 \in  W^{3+\nu, 1} \cap H^s \cap H^{-1}, \,  b_0 \in H^{s+1} \cap W^{4+\nu,1}$, with $\nu>0$.
The result of \cite{WanChen2016} extends \cite{Elgindi2015} to more general initial data. The 2D Boussinesq system with a damping term in the velocity equation \eqref{eq:system-vel}-\eqref{eq:vorticity-system_full} is studied in two recent articles. 
In \cite{CastroCordoba2019}, global existence in $H^s, s > 14$ and decay rates in $H^4$ of the solutions $(b, \uu)$ to system \eqref{eq:system-vel} in the periodic strip with no-flux conditions on the horizontal boundaries are obtained. 

Our true starting point is \cite{Wan2019}, where global existence of solutions $(b, \omega)$ s.t. $\omega \in \dot H^s \cap \dot H^{-2}$, $b \in \dot H^{s+1} \cap \dot H^{-1}$ with $s \ge 5$ to the 2D Boussinesq system in vorticity-stream formulation \eqref{eq:vorticity-system_full} has been proven in the whole space $\R^2$. A similar approach has been used in \cite{TWZZX2020} for the viscous case.

We write here the linear part of system \eqref{eq:vorticity-system_full},

\begin{equation}\label{eq:vorticity-system}
\begin{cases}
\partial_t b + N^2 \Delta^{-1} \partial_x \omega & =0, \\
\partial_t \omega + \alpha \omega - \partial_x b &=0.
\end{cases}
\end{equation}

Now define the following change of variable:

\begin{equation}\label{eq:change-variables}
\Omega:=N(-\Delta)^{-1/2}\omega.
\end{equation}

This way, system \eqref{eq:vorticity-system} reads:
\begin{equation} \label{eq:system-new-omega}
\begin{cases}
\partial_t b - N (-\Delta)^{-1/2} \partial_x \Omega&=0,\\
\partial_t \Omega - N (-\Delta)^{-1/2} \partial_x b &= - \alpha \Omega.
\end{cases}
\end{equation}
The equations in frequency variable read:
\begin{equation}\label{eq:system-Fourier}
\begin{cases}
\partial_t \widehat b - \frac{iN\xi_1}{|\xi|} \widehat \Omega&=0, \\
\partial_t \widehat \Omega - \frac{iN\xi_1}{|\xi|} \widehat b &=-\alpha \widehat \Omega,
\end{cases}
\end{equation}
and the eigenvalues are given by:
\begin{align}\label{eq:eigenvalues}
\lambda_{\pm}&=\dfrac{\alpha}{2} \pm \dfrac{1}{2} \sqrt{\alpha^2-\dfrac{4 N^2 \xi_1^2}{|\xi|^2}}.
\end{align}

We can compare system \eqref{eq:system-new-omega} with the first order formulation of the wave equation with damping
\begin{equation}\label{eq:wave-damo}
\begin{cases}
\partial_t b + \partial_x \Omega & =0, \\
\partial_t \Omega + \partial_x b &=-\alpha \Omega,
\end{cases}
\end{equation}
which is just a special case of the general class of weakly dissipative hyperbolic systems in conservative-dissipative form, which have been studied for instance in \cite{BHN2007}, i.e. multidimensional hyperbolic systems that can be written in the following form: 
\begin{equation}\label{eq:dissip}
\partial_t U + \sum_{i=1}^d A_i(U)\partial_{x_i} U = G(U),
\end{equation}
where $U\in\mathbb{R}^k$ depends on $t\geq 0$ and $x\in \mathbb{R}^d$, $A_i(U)$ are smooth symmetric $k\times k$-matrices and $G(U)$ is a smooth source term. Consider a constant equilibrium value $\bar U$ for system \eqref{eq:dissip}, i.e. such that $G(\bar U)=0$. A system is in conservative-dissipative form if there exists an integer $0<m< k$ and a positive definite $(k-m)\times (k-m)$ matrix $D$ such that:
\begin{equation}\label{eq:dissip1}
G'(\bar U)=\begin{pmatrix}
0 & 0 \\
0 & -D
\end{pmatrix}.
\end{equation}
It is well known that to obtain global existence results of smooth solutions, at least for small perturbations of a constant equilibrium state, for this class of systems, some supplementary conditions are needed to guarantee a sufficient coupling between the $m$-dimensional conservative and the $(k-m)$-dimensional dissipative part of the system. As discussed in detail in \cite{HN, yong, SK}, it is possible to obtain those results under the so-called Shizuta-Kawashima condition, here stated in a quite unusual form:\\

[SK] Given an equilibrium value $\bar U$ for system \eqref{eq:dissip}, and set $$A(\bar U, \xi)=\sum_{i=1}^d A_i(U)\xi_i=\begin{pmatrix}
A_{11}(\bar U, \xi) & A_{12}(\bar U, \xi) \\
A_{21}(\bar U, \xi) & A_{22}(\bar U, \xi)
\end{pmatrix}.$$
Under the assumption \eqref{eq:dissip1}, the Shizuta-Kawashima condition holds if, for all fixed $\xi\in \mathbb{R}^d\backslash \{0\}$ and every eigenvector $W\in \mathbb{R}^m$ of the symmetric matrix $A_{11}(\bar U, \xi)$ , we have that $A_{21}(\bar U, \xi) W\neq 0$. 

We cannot directly apply this approach to system \eqref{eq:system-Fourier}, since in this case we are dealing with a 0-order operator, and the functional framework is quite different. In particular, unlike the case of partially dissipative hyperbolic systems of first order, where the low frequency regime is represented by the heat kernel $e^{-t|\xi|^2}$, so that derivatives decay faster in time, here we can see from \eqref{eq:eigenvalues} (and a further discussion below) that the kernel behaves like $e^{-{t|\xi_1|^2}/{|\xi|^2}}$ when the first normalized component $\xi_1/|\xi|$ of the frequency variable $\xi=(\xi_1, \xi_2)$ is small. This implies that in our case only horizontal derivatives enhance time decay, but at the price of regularity.  Since the dispersion relation \eqref{eq:eigenvalues} is order zero, indeed, it is the direction rather than the modulus of the wavelength which governs the wave dynamics. This discussion will be clear in Section \ref{Sect3}. However, recalling the approach of \cite{KZuazua2011} we can notice that, at least formally, condition [SK] does not apply to system \eqref{eq:system-Fourier}, since the term $- \frac{iN\xi_1}{|\xi|} $ can vanish on the one-dimensional manifold $\xi_1=0$. This situation is analogous to the framework proposed in \cite{KZuazua2011} to deal with the cases where condition [SK] fails. In that paper, the analysis of the solutions to the linearized system \eqref{eq:dissip} was established precisely in the case where $A_{21}(\bar U, \xi) W = 0$ on a submanifold of $
\{\xi\in\mathbb{R}^d\backslash \{0\}\}$ of zero measure. In \cite{KZuazua2011}, a clever strategy based on the Kalman Rank Condition and the construction of an explicit Lyapunov functional was implemented to show the time decay rates for a suitable decomposition of the solutions to the linearized equation. 
Although in the present paper we cannot use this analysis, we can work out some estimates which are inspired by those in \cite{KZuazua2011}, to characterize the asymptotic behavior of smooth solutions to system \eqref{eq:vorticity-system_full}. 

Our result is stated below.

\begin{theorem}[Decay rates for the nonlinear system]\label{thm:decay-nonlinear-intro} 
Let $s \ge 5$, $\nu_0 >0$ and $\eps_0>0$ small enough. 
Let $(b(t), \omega (t))$ be the unique solution to system \eqref{eq:vorticity-system_full}, with initial data $(b_0, \omega_0)$ where $b_0 \in W^{2+\nu_0+s,1}(\R^2) \cap \dot H^{-1}(\R^2)  \subset \dot H^{s+2}(\R^2)  \cap \dot H^{-1}(\R^2) $ and $\omega_0 \in W^{1+\nu_0 + s, 1}(\R^2)  \cap \dot H^{-2} \subset \dot H^{s+1}(\R^2)  \cap \dot H^{-2}(\R^2)$. Introduce 
\begin{align*}
\tilde{\mathcal{E}}_0:=\|b_0\|_{\dot H^{-1} \cap \dot H^{s+1}} + \|\omega_0\|_{\dot H^{-2} \cap \dot H^s}, \quad  \mathcal{E}_0:=\|b_0\|_{W^{2+\nu_0 + s, 1} \cap \dot H^{-1}} + \|\omega_0\|_{W^{1+\nu_0 + s, 1} \cap \dot H^{-2}},
\end{align*}
which satisfy $\tilde{\mathcal{E}_0} \le C \mathcal{E}_0$, where $C>0$ is the Sobolev embedding constant. We assume that
\begin{align*}
\mathcal{E}_0 < C^{-1} \eps_0.
\end{align*}
Then, for $0 \le \sigma \le s-2$, the following nonlinear decay estimates hold true:
\begin{align*}
 \|b(t)\|_{H^{\sigma+1}} & \lesssim t^{-\frac 1 4} \mathcal{E}_0, \quad \|\omega(t)\|_{H^\sigma}  \lesssim t^{-\frac 3 4} \mathcal{E}_0.
\end{align*}
\begin{align*}
\|\de_x b(t)\|_{H^{\sigma}}  & \lesssim t^{-\frac 3 4} \mathcal{E}_0, \quad \|\de_x \omega(t)\|_{H^{\sigma-1}}\lesssim t^{-\frac 5 4} \mathcal{E}_0.
\end{align*}
\end{theorem}

\begin{remark}[On the initial data and their size]
The existence of global in time smooth solutions to system \eqref{eq:vorticity-system_full} is proved in [Theorem 1.1, \cite{Wan2019}]. To obtain the decay estimates of Theorem \ref{thm:decay-nonlinear-intro}, we rely on that result (see also Theorem \ref{thm:wan}, where we state [Theorem 1.1, \cite{Wan2019}] in terms of $(b, \Omega)$, where $\Omega$ is defined in \eqref{eq:change-variables}). 
Therefore, in the statement of Theorem \ref{thm:decay-nonlinear-intro} and thereafter, we choose $\eps_0>0$ small enough such that \cite{Wan2019} applies. We assume that for $s \ge 5$ 
the initial data $(b_0, \omega_0)$ satisfy 
\begin{align*}
\tilde{\mathcal{E}}_0:=\|b_0\|_{\dot H^{-1} \cap \dot H^{s+1}} + \|\omega_0\|_{\dot H^{-2} \cap \dot H^s} < \eps_0,
\end{align*}
which implies that there exists a unique smooth solution $(b, \omega)$ to system  \eqref{eq:vorticity-system_full} such that, in particular, 
\begin{align*}
\sup_t \|b (t)\|_{\dot H^{-1} \cap \dot H^{s+1}} + \| \omega (t) \|_{\dot H^{-2} \cap \dot H^s} \le \tilde{\mathcal{E}}_0 \le \eps_0.
\end{align*}
As explained in Remark \ref{rmk-thm}, the hypotheses on the initial data of Theorem \ref{thm:decay-nonlinear-intro} are slightly stronger than the setting of \cite{Wan2019}.
More precisely, for a given $\eps_0$ (for which [Theorem 1.1, \cite{Wan2019}] applies) and for any $\nu_0 >0$, we consider the initial data $(b_0, \omega_0) \in (W^{2+\nu_0+s,1} \cap \dot H^{-1}, W^{1+\nu_0+s} \cap \dot H^{-2})$ such that 
\begin{align*}
\tilde{\mathcal{E}}_0 \le C \mathcal{E}_0=\|b_0\|_{W^{2+\nu_0 + s, 1} \cap \dot H^{-1}} + \|\omega_0\|_{W^{1+\nu_0 + s, 1} \cap \dot H^{-2}} < \eps_0.
\end{align*}
\end{remark}

\subsection*{Plan of the paper}
First, in Section \ref{Sect2}, we perform some preliminary local expansions of the Green function in the Fourier space. Section \ref{Sect3} is devoted to the time decay estimates in the linear case, while Section \ref{Sect4} contains the analysis of the nonlinear system \eqref{eq:vorticity-system_full}.


\section{Eigenvalues expansion and orthogonal projectors}\label{Sect2}
Consider now system \eqref{eq:system-new-omega} with the eigenvalues given  by \eqref{eq:eigenvalues}.
We provide the eigenvalues expansion. Let $\mu:=\frac{\xi_1}{|\xi|}=\cos(\theta)$ for some $\theta \in [0, 2\pi]$. 

\subsection{The case of slow decay: $|\theta-k \pi/2| << 1, \quad k= 1,3$}
Performing a Taylor expansion in this case, namely for $\mu \simeq 0$, one gets the following expressions:

\begin{equation}\label{eq:eig.expansion}
\lambda_+=\alpha-\frac{N^2}{\alpha} \mu^2 + O(\mu^3), \quad \lambda_-=\frac{N^2}{\alpha} \mu^2+O(\mu^3).
\end{equation}

A straightforward computation provides the related eigenvectors:
\begin{equation}\label{eq:eigenvectors}
\mathcal{V}_+=\begin{pmatrix}
N \xi_1\\
i|\xi|\left(\alpha-\frac{N^2 \xi_1^2}{\alpha |\xi|^2}\right)
\end{pmatrix}, \quad 
\mathcal{V}_-=\begin{pmatrix}
\alpha |\xi|\\
iN\xi_1
\end{pmatrix}.
\end{equation}

To obtain explicit semigroup estimates of the linear system, we follow the classical Perturbation Theory by Kato \cite{Kato76}, which was adapted to partially dissipative hyperbolic systems by the authors of \cite{BHN2007} (see also \cite{BianchiniSIMA2018} for explicit computations in the context of singular approximation problems). In the same spirit, the expansions of the eigenprojectors of system \eqref{eq:system-new-omega} are developed below.

Consider the Green kernel $\widehat \Gamma (t, \xi)$ associated with \eqref{eq:system-Fourier}, which satisfies
\begin{equation}\label{eq:system-Fourier-compact}
\begin{aligned}
\ddt \widehat \Gamma & = \left(-B-i\frac{\xi_1}{|\xi|} A\right) \widehat \Gamma, \\
\widehat \Gamma (0, \xi)& =Id,
\end{aligned}
\end{equation}

where

\begin{equation}\label{def:A-B}
A=\begin{pmatrix}
0 & -N\\
-N & 0
\end{pmatrix}, \quad 
B=\begin{pmatrix}
0 & 0 \\
0 & \alpha
\end{pmatrix}.
\end{equation}

Recalling that $\mu=\frac{\xi_1}{|\xi|}$, define $z=i\mu$ and consider the entire function

\begin{equation}\label{def:E}
E(z)=B+zA=\begin{pmatrix}
0 & -N z \\
-Nz& \alpha
\end{pmatrix}.
\end{equation}

At the formal level, one has that

\begin{align*}
\widehat{\Gamma}(t, \xi)=e^{-E(i\xi_1/|\xi|)t} = \sum_{n=0}^\infty  {\frac{(-1)^n}{n!}} \left(-B-\frac{i\xi_1}{|\xi|} A\right)^n.
\end{align*}

If $z$ is not an exceptional point of the complex plane (see \cite{Kato76}), then the following decomposition holds true:

\begin{equation}\label{def:E}
E(z)=\lambda_+(z) \mathbb{P}_+(z) + \lambda_-(z) \mathbb{P}_-(z),
\end{equation}

where $\lambda_\pm(z)$ in \eqref{eq:eig.expansion} are the eigenvalues of $E(z)$ and $\mathbb{P}_\pm(z)$ are the related eigenprojectors. It can be easily seen that the only exceptional point of $E(z)$ is $z=0$, which is considered here.

The eigenprojector for the eigenvalue of $E(z)$ which vanishes in the regime where $z \simeq 0$ is given by

\begin{align}\label{def:eigenprojector}
\mathbb{P}(z)=-\frac{1}{2\pi i} \oint_{|z| \ll 1} (-E(z)-\zeta Id)^{-1} \, d\zeta,
\end{align}

where $\zeta \in \mathbb{C}$ is a complex number. One can expand as follows:

\begin{align*}
\mathbb{P}(z)=Q_0+\sum_{n \ge 1} z^n {P^{(n)}}, 
\end{align*}

where 

\begin{align*}
Q_0=Id-Q_-=\begin{pmatrix}
1 & 0\\
0 & 0
\end{pmatrix}
\end{align*}

is the projector into the kernel of $B$, i.e. $-E(0)-\zeta Id =-B-\zeta Id$, while

\begin{align*}
P^{(n)}=-\frac{1}{2\pi i}  \oint R^{(n)}, \quad n \ge 1, \quad R^{(n)}=(-B-\zeta Id)^{-1} (A (-B-\zeta Id)^{-1})^n.
\end{align*}

Notice that 

\begin{align*}
A(-B-\zeta Id)^{-1}=\begin{pmatrix}
0 & \frac{N}{(\alpha+\zeta)} \\
\frac{N}{\zeta} & 0
\end{pmatrix}.
\end{align*}

Explicitly, one has that

\begin{align*}
R^{(1)}=\frac{-N}{\zeta(\alpha+\zeta)} 
\begin{pmatrix}
0 & 1\\
1 & 0
\end{pmatrix}, \quad 
R^{(2)}=\frac{N^2}{\zeta(\alpha+\zeta)}(-B-\zeta Id)^{-1}.
\end{align*}

By using Cauchy Integral Formula, one obtains 
\begin{align*}
P_1=-\dfrac{1}{2\pi i} \oint_{|\xi |=\varepsilon} R^{(1)}= \begin{pmatrix}
0 & \frac{N}{\alpha} \\
\frac{N}{\alpha} & 0
\end{pmatrix}, 
\quad
P_2=-\dfrac{1}{2\pi i} \oint_{|\xi |=\varepsilon} R^{(2)}=\begin{pmatrix}
-\frac{N^2}{\alpha^2} & 0 \\
0 & \frac{N^2}{\alpha^2}
\end{pmatrix}.
\end{align*}

Thus, the second order expansion of $\mathbb{P}(z)$ reads

\begin{align*}
\mathbb{P}\left(\frac{i\xi_1}{|\xi|}\right)&=Q_0+
\frac{i\xi_1}{|\xi|} 
P_1 + \left(\frac{i\xi_1}{|\xi|} \right)^2 P_2
=\begin{pmatrix}
1+\frac{N^2}{\alpha^2} \frac{\xi_1^2}{|\xi|^2} & \frac{iN}{\alpha} \frac{\xi_1}{|\xi|} \\
\frac{iN}{\alpha} \frac{\xi_1}{|\xi|} & -\frac{N^2}{\alpha^2} \frac{\xi_1^2}{|\xi|^2}
\end{pmatrix}.
\end{align*}

On the other hand, the expansion for the orthogonal projector of the eigenvalue far from 0 is given by

\begin{align*}
\mathbb{P}_-\left(\frac{i\xi_1}{|\xi|}\right)=Id-\mathbb{P}\left(\frac{i\xi_1}{|\xi|}\right)=
\begin{pmatrix}
-\frac{N^2}{\alpha^2} \frac{\xi_1^2}{|\xi|^2} & -\frac{iN}{\alpha} \frac{\xi_1}{|\xi|} \\
-\frac{iN}{\alpha} \frac{\xi_1}{|\xi|} &\frac{N^2}{\alpha^2} \frac{\xi_1^2}{|\xi|^2}
\end{pmatrix}.
\end{align*}

This way, in the regime where $|\theta - k\pi/2| \ll 1, \; k=1,3$, one has that

\begin{align}
\widehat K (t,\xi)&=e^{-E\left(\frac{i\xi_1}{|\xi|}\right)t}=e^{-E(i\cos(\theta))t} \notag\\
&= \begin{pmatrix}
1+\frac{N^2}{\alpha^2} \cos^2(\theta) & \frac{iN}{\alpha} \cos(\theta) \\
\frac{iN}{\alpha} \cos(\theta) & -\frac{N^2}{\alpha^2} \cos^2(\theta)
\end{pmatrix} e^{-\frac{N^2}{\alpha} \cos^2(\theta) t} +  \begin{pmatrix}
-\frac{N^2}{\alpha^2} \cos^2(\theta) & -\frac{iN}{\alpha} \cos(\theta) \\
-\frac{iN}{\alpha} \cos(\theta) & \frac{N^2}{\alpha^2} \cos^2(\theta)
\end{pmatrix} e^{-\frac \alpha 2 t}. \label{eq:spectral-exp}
\end{align}


\subsection{The case of fast decay}
The complementing situation happens when $|\mu|=\frac{|\xi_1|}{|\xi|} \simeq 1$. In that case, the eigenvalue expansions are given by

\begin{align}
\lambda_+=\frac{\alpha}{2} + \frac{\sqrt{\alpha^2-4N^2}}{2}, \quad \lambda_-=\frac{\alpha}{2} - \frac{\sqrt{\alpha^2-4N^2}}{2},
\end{align}

while the eigenvectors read

\begin{align}\label{eq:eigenvectors-high-freq}
\mathcal{V}_+=\begin{pmatrix}
\frac{1}{2}(\alpha-\sqrt{\alpha^2-4N^2})\\
iN
\end{pmatrix}, \quad 
\mathcal{V}_-=\begin{pmatrix}
\frac{1}{2}(\alpha+\sqrt{\alpha^2-4N^2})\\
iN
\end{pmatrix}.
\end{align}

In the regime where $\frac{|\xi_1|}{|\xi|} \simeq 1$, 
\begin{align}\label{eq:spectral-exp1}
\widehat{\mathcal{K}}(t,\xi)=e^{-E\left(\frac{i\xi_1}{|\xi|}\right)t}\simeq (1+t) e^{-\frac{\alpha}{2}t} 
\begin{pmatrix}
 e^{-\frac{t\sqrt{\alpha^2-4N^2}}{2}}+ C_{11} e^{\frac{t\sqrt{\alpha^2-4N^2}}{2}} & e^{-\frac{t\sqrt{\alpha^2-4N^2}}{2}}+C_{12} e^{\frac{t\sqrt{\alpha^2-4N^2}}{2}}\\
e^{-\frac{t\sqrt{\alpha^2-4N^2}}{2}}+C_{21}e^{\frac{t\sqrt{\alpha^2-4N^2}}{2}} & e^{-t\sqrt{\alpha^2-4N^2}}+C_{22} e^{t\sqrt{\alpha^2-4N^2}}
\end{pmatrix},
\end{align}

for some constants $C_{ij} \in \mathbb{C}$, with $i, j \in \{1,2\}$.
The following lemma is then proven.

\begin{lemma}\label{lem-decomposition}
Let $\Gamma(t,x)=\mathcal{F}^{-1}(e^{-E(i\frac{\xi_1}{|\xi|})t})$ be the Green function associated to system \eqref{eq:system-Fourier}. The following decomposition holds:
\begin{align}
\Gamma(t,x)=K(t,x)+\mathcal{K}(t,x),
\end{align}
where $K(t,x), \, \mathcal{K}(t,x)$ are defined in \eqref{eq:spectral-exp}-\eqref{eq:spectral-exp1}.
\end{lemma}


\section{Decay estimates of the linear system}\label{Sect3}

We now provide the decay rates for the linearized system. We state the following simple but useful results.

\begin{lemma}\label{lem-initial-data}

Let $r > 0$,  $0<\nu < 1$ and $g \in W^{1+\nu+r,1}(\R^2).$ Then $g \in \dot H^{r+\tilde \nu}(\R^2)$ for any $\tilde \nu < \nu$.
\end{lemma}
\begin{proof}
If $g \in W^{1+\nu+r,1}(\R^2)$, then, in polar coordinates,
\begin{align}\label{eq:Ws1-bound}
| \widehat{g}(\theta, \rho)| \lesssim \dfrac{\|g\|_{W^{1+\nu+r}}}{1+\rho^{1+\nu+r}}. 
\end{align}
This implies that
\begin{align*}
\|g\|^2_{\dot H^{r+\tilde \nu}}&= \int_{\R^2} |\xi|^{2r+2\tilde \nu } |\widehat{g}|^2 \, d\xi  \lesssim \|g\|^2_{W^{1+\nu+r,1}} \int_0^{2\pi} \int_0^\infty \frac{\rho^{2r+2\tilde \nu +1}}{(1+\rho^{1+\nu+r})^2} \, d\theta \, d\rho \lesssim \|g\|^2_{W^{1+\nu+r,1}}.
\end{align*}
\end{proof}

\begin{lemma}\label{lem-Hs}
Let $r>0$ and $g \in \dot H^r(\R^2) \cap \dot H^{-1}(\R^2)$. Then $ g \in H^r(\R^2)$. 
\end{lemma}

\begin{proof}
Let $g \in \dot H^r(\R^2) \cap \dot H^{-1}(\R^2)$ and compute

\begin{align*}
\|g\|^2_{H^r} = \int_{\R^2} (1+|\xi|^2)^r |\widehat{g}|^2 \, d\xi & = \int_{|\xi| \le 1} (1+|\xi|^2)^r |\widehat{g}|^2 \, d\xi + \int_{|\xi| \ge 1} (1+|\xi|^2)^r |\widehat{g}|^2 \, d\xi \\
& \lesssim \|g\|^2_{\dot H^{-1}} + \|g\|^2_{\dot H^r}.
\end{align*}
\end{proof}

The two results above summarize as follows: 
\begin{align}\label{eq:inclusion}
W^{1+\nu+r,1}(\R^2) \cap \dot H^{-1}(\R^2) \subset \dot H^r(\R^2) \cap \dot H^{-1}(\R^2) \subset H^r(\R^2)\cap \dot H^{-1}(\R^2). 
\end{align}


The following decay estimates are provided below. 

\begin{theorem}[Decay estimates for the linearized system]
\label{thm-decay-linear}
Consider system \eqref{eq:system-new-omega} in $\R^2$, endowed with initial data $\mathcal{U}_0=(b_0, \Omega_0) \in W^{1+\nu+r,1}$, for any $r, \nu \ge 0$. The following decay estimates hold true.
\begin{align}\label{eq:estimate_Hs}
\|b(t)\|_{H^r} & \lesssim t^{-1/4} \|b_0\|_{W^{1+\nu+r,1}}+t^{-3/4}  \|\Omega_0\|_{W^{1+\nu+r,1}},\notag\\ \|\de_x b(t)\|_{H^{r-1}} & \lesssim t^{-3/4}\|b_0\|_{W^{1+\nu+r,1}}+t^{-5/4}\|\Omega_0\|_{W^{1+\nu+r,1}}, \notag\\ 
\|\de_y b(t)\|_{H^{r-1}} & \lesssim t^{-1/4}\|b_0\|_{W^{1+\nu+r,1}}+t^{-3/4}\|\Omega_0\|_{W^{1+\nu+r,1}},\notag\\\|\Omega(t)\|_{H^r} &\lesssim t^{-3/4}\|b_0\|_{W^{1+\nu+r,1}}+t^{-5/4} \|\Omega_0\|_{W^{1+\nu+r,1}},
\notag\\
\|\de_x \Omega(t)\|_{H^{r-1}} & \lesssim t^{-5/4}\|b_0\|_{W^{1+\nu+r,1}}+t^{-7/4}\|\Omega_0\|_{W^{1+\nu+r,1}}, \notag\\
 \|\de_y \Omega(t)\|_{H^{r-1}} & \lesssim t^{-3/4}\|b_0\|_{W^{1+\nu+r,1}}+t^{-5/4}\|\Omega_0\|_{W^{1+\nu+r,1}}.\notag\\
\end{align}
\end{theorem}
In the course of the proof, we will rely on the result below.
\begin{lemma}[\cite{Elgindi2017}, Lemma 2.1]
\begin{align}\label{lem-tarek}
\int_0^{2\pi} |\cos(\theta)|^k e^{-(\cos^2\theta)t} \, d\theta \simeq c_k t^{-(1+k)/2} \quad \text{as   } t \rightarrow \infty.
\end{align}
\end{lemma}
\begin{proof}
A suitable change of variables, which can be found for instance in \cite{KZuazua2011, Elgindi2017}, allows to conclude.
\end{proof}


\begin{proof}[Proof of Theorem \ref{thm-decay-linear}]
We start with the $H^r$-estimate. The proof is divided in three main steps:
\begin{itemize}
\item first, we apply the decomposition given by Lemma \ref{lem-decomposition};
\item next, we pass to polar coordinates and use inequality \eqref{eq:Ws1-bound};
\item finally, we apply Lemma \ref{lem-tarek}.
\end{itemize}
\begin{align*}
\|b(t)\|_{H^r}^2 &\lesssim \int_{\R^2} (1+|\xi|^{2})^r e^{-\frac{CN^2}{\alpha^2} \frac{\xi_1^2}{|\xi|^2}t} \left[|\widehat b_0(\xi)|^2 + \frac{\xi_1^2}{|\xi|^2} |\widehat{\Omega}_0(\xi)|^2\right] \, d\xi\\
& \lesssim \int_0^{2\pi} \int_0^\infty (1+\rho^2)^r \rho \, e^{-\frac{CN^2}{\alpha^2}(\cos^2 \theta) t} \left[ |\widehat b_0(\theta, \rho)|^2 + \cos^2(\theta) |\widehat{\Omega}_0 (\theta, \rho)|^2 \right] \, d\theta \, d\rho\\
& \lesssim  \|b_0\|^2_{W^{1+\nu+r,1}} \int_0^{2\pi} e^{-\frac{CN^2}{\alpha^2}(\cos^2 \theta) t} \, d\theta \int_0^\infty  \frac{(1+\rho^2)^r \rho}{(1+\rho^{1+\nu+r})^2} \, d\rho \\
& \quad + \|\Omega_0\|^2_{W^{1+\nu+r,1}} \int_0^{2\pi} (\cos^2\theta) \, e^{-\frac{CN^2}{\alpha^2}(\cos^2 \theta) t} \, d\theta \int_0^\infty  \frac{(1+\rho^2)^r \rho}{(1+\rho^{1+\nu+r})^2} \, d\rho \\
&\lesssim t^{-1/2} \|b_0\|^2_{ W^{1+\nu+r,1}} + t^{-3/2} \|\Omega_0\|^2_{ W^{1+\nu+r,1}}.
\end{align*}

Next, we consider the derivatives.
\begin{align*}
\|\de_x b(t)\|_{H^{r-1}}^2 &\lesssim \int_{\R^2} |\xi_1|^2 (1+|\xi|^{2})^{r-1} e^{-\frac{CN^2}{\alpha^2} \frac{\xi_1^2}{|\xi|^2}t} \left[|\widehat b_0(\xi)|^2 + \frac{\xi_1^2}{|\xi|^2} |\widehat{\Omega}_0(\xi)|^2\right] \, d\xi\\
& \lesssim \int_0^{2\pi} \int_0^\infty (1+\rho^2)^{r-1} \rho^3 \, (\cos^2 \theta) \, e^{-\frac{CN^2}{\alpha^2}(\cos^2 \theta) t} \left[ |\widehat b_0(\theta, \rho)|^2 + (\cos^2\theta) |\widehat{\Omega}_0 (\theta, \rho)|^2 \right] \, d\theta \, d\rho\\
& \lesssim  \|b_0\|^2_{W^{1+\nu+r,1}} \int_0^{2\pi} (\cos^2\theta) e^{-\frac{CN^2}{\alpha^2}(\cos^2 \theta) t} \, d\theta \int_0^\infty  \frac{(1+\rho^2)^{r-1} \rho^3}{(1+\rho^{1+\nu+r})^2} \, d\rho \\
& \quad + \|\Omega_0\|^2_{W^{1+\nu+r,1}} \int_0^{2\pi} (\cos^4\theta) \, e^{-\frac{CN^2}{\alpha^2}(\cos^2 \theta) t} \, d\theta \int_0^\infty  \frac{(1+\rho^2)^{r-1} \rho^3}{(1+\rho^{1+\nu+r})^2} \, d\rho \\
&\lesssim t^{-3/2} \|b_0\|^2_{ W^{1+\nu+r,1}} + t^{-5/2} \|\Omega_0\|^2_{ W^{1+\nu+r,1}}.
\end{align*}
The remaining estimates are analogous. Notice that since the dispersion relation (the eigenvalues) is homogeneous of degree zero, the $x$-derivative enhances time-decay, at the price of regularity.
\end{proof}


\section{Decay estimates for the nonlinear system}\label{Sect4}
Now we deal with the nonlinear model, which in terms of the diagonalized variables reads as follows:
\begin{equation} \label{eq:system-nonlinear}
\begin{aligned}
\partial_t b - N (-\Delta)^{-1/2} \partial_x \Omega&=N^{-1}( (-\Delta)^{-1/2} \de_y\Omega)\de_x b - N^{-1}( (-\Delta)^{-1/2} \de_x\Omega) \de_y b,\\
\partial_t \Omega - N (-\Delta)^{-1/2} \partial_x b +  \alpha \Omega &=N^{-1}( (-\Delta)^{-1/2} \de_y\Omega)\de_x \Omega - N^{-1}( (-\Delta)^{-1/2} \de_x\Omega) \de_y \Omega\\
&\quad + N^{-1}[(-\Delta)^{-1/2}, (-\Delta)^{-1/2} \de_y \Omega] (-\Delta)^{1/2}  \de_x\Omega\\
&\quad - N^{-1} [(-\Delta)^{-1/2}, (-\Delta)^{-1/2} \de_x \Omega] (-\Delta)^{1/2} \de_y\Omega.
\end{aligned}
\end{equation}

The global in time well-posedness of system \eqref{eq:system-nonlinear} has been established in \cite{Wan2019} in the framework of homogeneous Sobolev spaces $\dot H^s \cap \dot H^{-1}$ with $s\ge 6$ in the whole $\R^2$ domain.
The same type of results for non-homogeneous Sobolev spaces (in terms of the velocity variable) in a bounded domain (the periodic strip $\mathbb{T} \times [-1,1]$ with no-slip conditions) can be found in \cite{CastroCordoba2019} and are based on a different strategy.

Starting from the results of \cite{Wan2019}, where the global in time well-posedness of solutions to system \eqref{eq:system-nonlinear} in homogeneous Sobolev spaces is obtained, here we provide explicit decay rates of the smooth solutions. We first state the global in time existence result due to Wan \cite{Wan 2019}.
\begin{theorem}[\cite{Wan2019}, Theorem 1.1]\label{thm:wan}
Let $s \ge 6$. Consider the initial data $(\Omega_0, b_0)$ such that $(b_0, \Omega_0) \in \dot{H}^{s}(\R^2)\cap \dot H^{-1}(\R^2)$. Denote by
\begin{align*}
\tilde{\mathcal{E}}_0=\|b_0\|_{\dot{H}^{s}(\R^2)\cap \dot H^{-1}(\R^2)} + \|\Omega_0\|_{\dot{H}^{s}(\R^2)\cap \dot H^{-1}(\R^2)}.
\end{align*}
There exists $\varepsilon_0 >0$ small enough such that, if $\tilde{\mathcal{E}}_0 \le \varepsilon_0$, then system (\ref{eq:system-nonlinear})
admits a unique global in time solution $\mathcal{U}(t)=(b(t), \Omega(t))  \in  \dot{H}^{s}(\R^2)\cap \dot H^{-1}(\R^2)$. In particular,
\begin{align*}
\sup_{ t} \|b(t)\|_{\dot{H}^{s}(\R^2)\cap \dot H^{-1}(\R^2)} +  \|\Omega(t)\|_{\dot{H}^{s}(\R^2)\cap \dot H^{-1}(\R^2)} \lesssim \tilde{\mathcal{E}}_0.
\end{align*}
\end{theorem}

Now we state our result in terms of the new unknown variables $(b, \Omega)$ of system \eqref{eq:system-nonlinear}. 
\begin{theorem}[Decay rates for the nonlinear system]\label{thm:decay-nonlinear} 
Let $s \ge 6, \nu_0 \ge 0$ and let $\eps_0>0$ as in the statement of Theorem \ref{thm:wan}. Consider the unique global in time solution $\mathcal{U}(t)=(b(t), \Omega(t))$ to system \eqref{eq:system-nonlinear} with initial data $(b_0, \Omega_0) \in W^{1+\nu_0 + s, 1} \cap \dot H^{-1}$, such that
\begin{align*}
\mathcal{E}_0:=\|b_0\|_{W^{1+\nu_0 + s, 1} \cap \dot H^{-1}} + \|\Omega_0\|_{W^{1+\nu_0 + s, 1} \cap \dot H^{-1}} < C^{-1} \eps_0,
\end{align*}
where $C>0$ as in Theorem \ref{thm:decay-nonlinear-intro}.
Then, for any $1 \le \sigma \le s-2$, the following nonlinear decay estimates hold true:
\begin{align*}
 \|b(t)\|_{H^\sigma} & \lesssim t^{-\frac 1 4} \mathcal{E}_0, \quad \|\Omega(t)\|_{H^\sigma}  \lesssim t^{-\frac 3 4} \mathcal{E}_0.
\end{align*}
\begin{align*}
\|\de_x b(t)\|_{H^{\sigma-1}}  & \lesssim t^{-\frac 3 4} \mathcal{E}_0, \quad \|\de_x \Omega(t)\|_{H^{\sigma-1}}\lesssim t^{-\frac 5 4} \mathcal{E}_0.
\end{align*}
\end{theorem}

\begin{proof}[Proof of Theorem \ref{thm:decay-nonlinear-intro}]
We just observe that Theorem \ref{thm:decay-nonlinear-intro} simply follows by Theorem \ref{thm:decay-nonlinear} and the change of variable $\Omega=N(-\Delta)^{-1/2} \omega$ in \eqref{eq:change-variables}.
\end{proof}

\begin{remark}\label{rmk-thm}
The global in time result of Theorem \ref{thm:wan}, which holds under the assumption $\tilde{\mathcal{E}}_0:=\|b_0\|_{\dot H^s \cap \dot H^{-1}} + \|\Omega_0\|_{\dot H^s \cap \dot H^{-1}} < \eps_0$ can be found in \cite{Wan2019}, while the decay rates are obtained in the present paper. The assumptions on the initial data stated in Theorem \ref{thm:decay-nonlinear} are slightly stronger than the ones of Theorem \ref{thm:wan}. In Theorem \ref{thm:decay-nonlinear} we take indeed $\mathcal{U}_0=(b_0, \Omega_0) \in W^{1+\nu+s,1} \cap \dot H^{-1} \subset \dot {H^s} \cap \dot H^{-1}$. This additional requirement is due to the fact that the space $W^{r,1}$ is crucial for the decay estimates of Theorem \ref{thm-decay-linear}. This hypothesis was also used in \cite{Elgindi2015}, while its necessity has been recently overcome in \cite{Wan2020} by means of an additional amount of technicality involving Strichartz estimates. 
\end{remark}

\begin{proof}[Proof of Theorem \ref{thm:decay-nonlinear}]
We write the Duhamel formula for system \eqref{eq:system-nonlinear}, which reads
\begin{align*}
\begin{pmatrix}
b(t)\\
\Omega(t)
\end{pmatrix}=\Gamma(t)\begin{pmatrix}
b_0(x)\\
\Omega_0(x)
\end{pmatrix} & + \int_0^t \Gamma(t-\tau)\begin{pmatrix}
N^{-1}((-\Delta)^{-1/2} \de_y\Omega)\de_x b- N^{-1}((-\Delta)^{-1/2} \de_x\Omega) \de_y b\\
N^{-1}((-\Delta)^{-1/2} \de_y\Omega)\de_x \Omega - N^{-1}((-\Delta)^{-1/2} \de_x \Omega) \de_y \Omega
\end{pmatrix}\\
& + \int_0^t \Gamma(t-\tau)\begin{pmatrix}
0\\
N^{-1}[(-\Delta)^{-1/2}, (-\Delta)^{-1/2} \de_y\Omega]  (-\Delta)^{1/2} \de_x\Omega
\end{pmatrix}\\
& - \int_0^t \Gamma(t-\tau)\begin{pmatrix}
0\\
N^{-1} [(-\Delta)^{-1/2}, (-\Delta)^{-1/2}  \de_x\Omega] (-\Delta)^{1/2} \de_y\Omega
\end{pmatrix}\\
&=:\Gamma(t)\begin{pmatrix}
b_0(x)\\
\Omega_0(x)
\end{pmatrix} + \int_0^t \Gamma(t-\tau)\begin{pmatrix}
S_b\\
S_\Omega
\end{pmatrix},
\end{align*}

where we recall from Lemma \ref{lem-decomposition} that, denoting by $\mathcal{R}_j=\frac{\xi_j}{|\xi|}$ the Riesz transform, the principal part of $\Gamma(t)$ is given by
\begin{align*}
\Gamma(t)\simeq\mathcal{F}^{-1} \begin{pmatrix}
e^{-\mathcal{R}_1^2(\xi)t}
 & |\mathcal{R}_1(\xi)|e^{-\mathcal{R}_1^2(\xi)t}
\\
|\mathcal{R}_1(\xi)| e^{-\mathcal{R}_1^2(\xi)t}
& |\mathcal{R}_1(\xi)|^2e^{-\mathcal{R}_1^2(\xi)t}
\end{pmatrix}.
\end{align*}

We apply the fractional derivative $D_\x^r$ for $r=0, \cdots, s$ and integrate in space, so that, using the estimates \eqref{eq:estimate_Hs} of Theorem \ref{thm-decay-linear} with
\begin{align}\label{ineq:nu}
\nu < \frac{\nu_0}{4},
\end{align}

so that

\begin{align*}
\begin{pmatrix}
\|D_\x^r b(t)\|_{L^2}\\
\|D_\x^r \Omega(t)\|_{L^2}
\end{pmatrix}
& \lesssim 
\begin{pmatrix}
\min\{1, t^{-1/4}\} \|b_0\|_{W^{1+\nu+r,1}}+\min\{1, t^{-3/4}\} \|\Omega_0\|_{W^{1+\nu+r,1}}\\
\min\{1, t^{-3/4}\} \|b_0\|_{W^{1+\nu+r,1}}+\min\{1, t^{-5/4}\} \|\Omega_0\|_{W^{1+\nu+r,1}}
\end{pmatrix}\\
&\quad+\int_0^t \begin{pmatrix}
\|D_\x^r [\Gamma_{11} (t-\tau)*S_b] \|_{L^2} + \| D_\x^r  [\Gamma_{12} (t-\tau)*S_\Omega] \|_{L^2}  \\
\|D_\x^r [\Gamma_{21} (t-\tau)*S_b] \|_{L^2} + \| D_\x^r [\Gamma_{22} (t-\tau)*S_\Omega] \|_{L^2}  \\
\end{pmatrix},
\end{align*}
and so we obtain

\begin{align*}
\begin{pmatrix}
\|D_\x^r b(t)\|_{L^2}\\
\|D_\x^r\Omega(t)\|_{L^2}
\end{pmatrix} & \lesssim \begin{pmatrix}
\min\{1, t^{-1/4}\} \|b_0\|_{W^{1+\nu+r,1}}+\min\{1, t^{-3/4}\} \|\Omega_0\|_{W^{1+\nu+r,1}}\\
\min\{1, t^{-3/4}\} \|b_0\|_{W^{1+\nu+r,1}}+\min\{1, t^{-5/4}\} \|\Omega_0\|_{W^{1+\nu+r,1}}\\
\end{pmatrix}\\
&\quad +\int_0^t \begin{pmatrix}
(t-\tau)^{-1/4} \|S_b\|_{W^{1+\nu+r,1}} + (t-\tau)^{-3/4} \|S_\Omega\|_{W^{1+\nu+r,1}} 
 \\
(t-\tau)^{-3/4} \|S_b\|_{W^{1+\nu+r,1}} + (t-\tau)^{-5/4} \|S_\Omega\|_{W^{1+\nu+r,1}} 
\end{pmatrix}.\\
\end{align*}
More explicitly, one has that
\begin{align*}
&\begin{pmatrix}
\|D_\x^r b(t)\|_{L^2}\\
\|D_\x^r \Omega(t)\|_{L^2}
\end{pmatrix} \lesssim \begin{pmatrix}
\min\{1, t^{-1/4}\} \|b_0\|_{W^{1+\nu+r,1}}+\min\{1, t^{-3/4}\} \|\Omega_0\|_{W^{1+\nu+r,1}}\\
\min\{1, t^{-3/4}\} \|b_0\|_{W^{1+\nu+r,1}}+\min\{1, t^{-5/4}\} \|\Omega_0\|_{W^{1+\nu+r,1}}
\end{pmatrix} \\
&\qquad +\int_0^t \begin{pmatrix}
(t-\tau)^{-1/4} \left[\|(-\Delta)^{-1/2} \de_y\Omega \cdot \de_x b\|_{W^{1+\nu+r,1}}+ \| (-\Delta)^{-1/2} \de_x\Omega \cdot \de_y b\|_{W^{1+\nu+r,1}}\right]\\
(t-\tau)^{-3/4} \left[\| (-\Delta)^{-1/2} \de_y\Omega \cdot \de_x b\|_{W^{1+\nu+r,1}}+\|(-\Delta)^{-1/2} \de_x  \Omega \cdot \de_y b\|_{W^{1+\nu+r,1}}\right]\\
\end{pmatrix} \quad (i)\\
&\qquad +\int_0^t \begin{pmatrix}
(t-\tau)^{-3/4} \left[\| (-\Delta)^{-1/2} \de_x\Omega \cdot \de_y \Omega\|_{W^{1+\nu+r,1}} +\| (-\Delta)^{-1/2} \de_y\Omega \cdot \de_x \Omega\|_{W^{1+\nu+r,1}}\right]\\
(t-\tau)^{-5/4} \left[\| (-\Delta)^{-1/2} \de_x\Omega \cdot \de_y \Omega\|_{W^{1+\nu+r,1}} + \| (-\Delta)^{-1/2} \de_y\Omega \cdot \de_x \Omega\|_{W^{1+\nu+r,1}}\right]
\end{pmatrix} \quad (ii)\\
&\qquad + \int_0^t \begin{pmatrix}
(t-\tau)^{-3/4} \|[(-\Delta)^{-1/2},(-\Delta)^{-1/2}  \de_y \Omega] (-\Delta)^{1/2}\de_x\Omega\|_{W^{1+\nu+r,1}} \\
(t-\tau)^{-5/4} \|[(-\Delta)^{-1/2}, (-\Delta)^{-1/2} \de_y \Omega] (-\Delta)^{1/2}\de_x\Omega\|_{W^{1+\nu+r,1}} \\
\end{pmatrix}\quad (iii)\\
&\qquad + \int_0^t \begin{pmatrix}
(t-\tau)^{-3/4} \|[(-\Delta)^{-1/2}, (-\Delta)^{-1/2}  \de_x\Omega] (-\Delta)^{1/2}\de_y\Omega\|_{W^{1+\nu+r,1}} \\
(t-\tau)^{-5/4} \|[(-\Delta)^{-1/2}, (-\Delta)^{-1/2} \de_x \Omega] (-\Delta)^{1/2}\de_y\Omega\|_{W^{1+\nu+r,1}}
\end{pmatrix}\quad (iv).
\end{align*}

We focus on the worst case, which is the second line (i). We apply Lemma \ref{lem-bilinear-estimate} in the Appendix, 
\begin{equation}\label{eq:first-line}
\begin{aligned}
\|\de_y (-\Delta)^{-1/2} \Omega \cdot \de_x b\|_{ W^{1+\nu+r, 1}} & \lesssim \|\de_y (-\Delta)^{-1/2}\Omega \|_{ H^{1+\nu+r}} \|\de_x b\|_{L^2} + \|\de_y (-\Delta)^{-1/2}\Omega \|_{L^2} \|\de_x b\|_{ H^{1+\nu+r}}.
\end{aligned}
\end{equation}

Now for $\sigma \ge 1$ define the functional
\begin{align}
\mathcal{M}_\sigma(t)&=\sup_{\tau \ge 1}  \{\tau^{1/4}\|b\|_{ H^\sigma} + \tau^{3/4} \|\Omega\|_{ H^\sigma} +\tau^{3/4}\|\de_x b\|_{ H^{\sigma-1}} + \tau^{5/4} \|\de_x \Omega\|_{ H^{\sigma-1}}\}.
\label{def:energy-funcional}
\end{align}

We appeal to the Interpolation Lemma \ref{lem-gagliardo}, see the Appendix below, where $s=1+\nu+r$. In order to optimize the regularity requirements, our strategy is to minimize with respect to $\theta$ in Lemma \ref{lem-gagliardo}. More precisely, using the functional defined above, notice that
\begin{align}\label{eq:estimate1}
 \|\de_y (-\Delta)^{-1/2}\Omega \|_{ H^{1+\nu+r}} & \lesssim  \|\de_y (-\Delta)^{-1/2}\Omega \|_{ H^{s_0}}^\theta  \|\de_y (-\Delta)^{-1/2}\Omega \|_{ H^{s_1}}^{1-\theta} \lesssim \tau^{-3\theta /4} \mathcal{E}_0^{(1-\theta)}\mathcal{M}_r^{\theta}(t).
\end{align}

Similarly, one obtains that

\begin{align}\label{eq:estimate2}
\|\de_x b\|_{H^{1+\nu+r}} \lesssim \tau^{-3/4 \theta}  \mathcal{E}_0^{(1-\theta)}\mathcal{M}_r^{\theta}(t).
\end{align}

Plugging all the computations above in the first component of (i), 

\begin{align*}
\int_0^t (t-\tau)^{-1/4} & \|\de_y (-\Delta)^{-1/2} \Omega \cdot \de_x b\|_{W^{1+\nu+r, 1}} \, d\tau  \lesssim \int_0^t \|\de_y (-\Delta)^{-1/2}\Omega \|_{ H^{1+\nu+r}} \|\de_x b\|_{L^2}\, d\tau\\
& \quad + \int_0^t  (t-\tau)^{-1/4} \|\de_y (-\Delta)^{-1/2}\Omega \|_{L^2} \|\de_x b\|_{ H^{1+\nu+r}} \, d\tau\\
& \lesssim  \int_0^t  (t-\tau)^{-1/4} \|\de_y (-\Delta)^{-1/2}\Omega \|_{ H^{s_0}}^\theta  \|\de_y (-\Delta)^{-1/2}\Omega \|_{ H^{s_1}}^{1-\theta} \|\de_x b\|_{L^2}\, d\tau\\
& \quad + \int_0^t  (t-\tau)^{-1/4} \|\de_y (-\Delta)^{-1/2}\Omega \|_{L^2} \|\de_x b\|_{H^{s_0}}^\theta \|\de_x b\|_{H^{s_1}}^{(1-\theta)}\, d\tau \\
& \lesssim \mathcal{M}^{1+\theta}_r (t) \, \mathcal{E}_0^{(1-\theta)}
\int_0^t  (t-\tau)^{-1/4} \tau^{-3/4(\theta+1) } \, d\tau \simeq t^{-\phi},
\end{align*}

where, according to Lemma \ref{lem-BHN}, $\phi=\min \left\{\frac 14, \frac{3 \theta}{4}+\frac 3 4,  \frac 1 4+\frac{3 \theta}{4}+\frac 3 4-1 \right\}$. We choose $\theta$ such that

\begin{align*}
\frac 3 4 \theta \ge \frac 1 4, \qquad \text{i.e.} \quad \theta \ge \frac 1 3.
\end{align*}

We apply Lemma \ref{lem-gagliardo} with $\theta=\frac 1 3$ and $s=1+\nu+r$ and $s_0=r$. Thus one has that 
\begin{align}\label{eq:s1}
s_1=\dfrac{3s-(r-1)}{2}=2+\frac 3 2 \nu + r.
\end{align}

Plugging \eqref{eq:estimate1} and \eqref{eq:estimate2} in \eqref{eq:first-line},
\begin{align*}
 \|\de_y (-\Delta)^{-1/2}\Omega \cdot \de_x b\|_{W^{1+\nu+r, 1}} & \lesssim \tau^{-1} \mathcal{E}_0^{\frac 2 3} \mathcal{M}_\sigma^{\frac 4 3}(t)
=  \tau^{-1} \mathcal{E}_0^{\frac 2 3} \mathcal{M}_r^{\frac 4 3}(t).
\end{align*}

Appealing to Lemma \ref{lem-BHN},
\begin{align*}
\int_0^t (t-\tau)^{-1/4} \|\de_y (-\Delta)^{-1/2} \Omega \cdot \de_x b\|_{ W^{1+\nu+r, 1}} \, d\tau & \lesssim t^{-\frac 14} \mathcal{E}_0^{\frac 2 3}\mathcal{M}_r^{\frac 4 3}(t).
\end{align*}

Similarly, we handle the second term of the first component of (i) as follows
\begin{align*}
\|\de_x (-\Delta)^{-1/2} \Omega \cdot \de_y b\|_{ W^{1+\nu+r, 1}} & \lesssim  \|\de_x (-\Delta)^{-1/2} \Omega\|_{ H^{1+\nu+r}} \|\de_y b\|_{L^2} + \|\de_x (-\Delta)^{-1/2} \Omega\|_{L^2} \|\de_y b\|_{ H^{1+\nu+r}}\\
& \lesssim  \|\de_x \Omega\|_{ H^{\nu+r}} \|\de_y b\|_{L^2} + \|\de_x (-\Delta)^{-1/2} \Omega\|_{L^2} \|\de_y b\|_{ H^{1+\nu+r}}\\
& \lesssim  \|\de_x \Omega\|_{ H^{r-1}}^{\theta}  \|\Omega\|_{ H^{{s}_1}}^{(1-\theta)} \|\de_y b\|_{L^2} + \|\de_x (-\Delta)^{-1/2}\Omega\|_{L^2}\|\de_y b\|_{ H^{r-1}}^{\tilde \theta} \|\de_y b\|_{ H^{\tilde{s}_1}}^{(1-\tilde \theta)}\\
&\lesssim \mathcal{E}_0^{(1-\theta)} \tau^{-\left(\frac 5 4 \theta + \frac 1 4\right) } \mathcal{M}_r^{1+\theta}(\tau)+  \mathcal{E}_0^{(1-\tilde \theta)} \tau^{-\left(\frac 5 4  + \frac 1 4 \tilde \theta \right) } \mathcal{M}_r^{1+\tilde \theta}(\tau).
\end{align*}

Applying the same procedure, we minimize among the $\theta$'s satisfying
$$\frac 1 4 + \frac 5 4 \theta + \frac 1 4 - 1 \ge \frac 1 4, \quad \theta \ge \frac 3 5.$$
We choose now $\theta = \frac 3 5$, so that ${s}_1=\frac 3 2 + \frac 5 2 \nu + r,$ while it is enough to set $\tilde \theta=\nu$, so that $\tilde{s}_1=\frac{1+2\nu}{1-\nu} + r <  1+ \nu_0 + r$ thanks to inequality \eqref{ineq:nu}. The second component of (i) gives
\begin{align*}
\int_0^t (t-\tau)^{-1/4} \|\de_x (-\Delta)^{-1/2} \Omega \cdot \de_y b\|_{H^r} \, d\tau & \lesssim \mathcal{E}_0^{\frac 2 5}  \mathcal{M}_r^{\frac 8 5}(t) \int_0^t (t-\tau)^{-1/4} \tau^{-1} \, d\tau\\
&   \lesssim t^{-\frac 1 4} \mathcal{E}_0^{\frac 2 5}\mathcal{M}_r^{\frac 8 5 }(t).
\end{align*}

The estimates of (ii)-(iv) are similar, we only sketch how to handle the first component of (iii).
\begin{align*}
& \int_0^t (t-\tau)^{-3/4} \|[(-\Delta)^{-1/2},(-\Delta)^{-1/2}  \de_y \Omega] (-\Delta)^{1/2}\de_x\Omega\|_{W^{1+\nu+r,1}} \\
& \quad = \int_0^t (t-\tau)^{-3/4} \left\|(-\Delta)^{-1/2}\left[(-\Delta)^{-1/2}  \de_y \Omega \cdot (-\Delta)^{1/2}\de_x\Omega \right]\right\|_{W^{1+\nu+r,1}} \quad (a)\\
& \qquad +  \int_0^t (t-\tau)^{-3/4} \| (-\Delta)^{1/2} \de_y \Omega \cdot \de_x\Omega \|_{W^{1+\nu+r,1}} \quad (b).
\end{align*}

The term (b) can be treated exactly as done before. We deal with (a), by applying Lemma \ref{lem-bilinear-estimate} stated in the Appendix.
\begin{align*}
& \int_0^t (t-\tau)^{-3/4} \left\|(-\Delta)^{-1/2}\left[(-\Delta)^{-1/2}  \de_y \Omega \cdot (-\Delta)^{1/2}\de_x\Omega \right]\right\|_{W^{1+\nu+r,1}}\\
& \quad = \int_0^t (t-\tau)^{-3/4} \left\|(-\Delta)^{-1/2}  \de_y \Omega \cdot (-\Delta)^{1/2}\de_x\Omega \right\|_{W^{\nu+r,1}}\\
& \quad \lesssim  \int_0^t (t-\tau)^{-3/4} \|(-\Delta)^{-1/2}  \de_y \Omega\|_{L^2} \|(-\Delta)^{1/2}\de_x\Omega\|_{H^{\nu+r}}\\
& \qquad +  \int_0^t (t-\tau)^{-3/4}  \|(-\Delta)^{-1/2}  \de_y \Omega\|_{H^{\nu+r}} \|(-\Delta)^{1/2}\de_x\Omega\|_{L^2}\\
& \quad \lesssim  \int_0^t (t-\tau)^{-3/4} \|(-\Delta)^{-1/2}  \de_y \Omega\|_{L^2} \|\de_x\Omega\|_{H^{1+\nu+r}}\\
& \qquad  +  \int_0^t (t-\tau)^{-3/4}  \| \de_y \Omega\|_{H^{-1+\nu+r}} \|(-\Delta)^{1/2}\de_x\Omega\|_{L^2}\\
& \quad \lesssim  \int_0^t (t-\tau)^{-3/4} \|(-\Delta)^{-1/2}  \de_y \Omega\|_{L^2} \|\de_x\Omega\|^\theta_{H^{r-1}} \|\de_x\Omega\|^{(1-\theta)}_{H^{s_1}}\\
& \qquad  + \int_0^t (t-\tau)^{-3/4}  \| \de_y \Omega\|^\theta_{H^{r-1}}  \| \de_y \Omega\|^{(1-\theta)}_{H^{s_1}} \|(-\Delta)^{1/2}\de_x\Omega\|_{L^2}\\
& \quad \lesssim t^{-\frac 3 4} \mathcal{E}_0^{\frac 2 3} \mathcal{M}_r^{\frac 4 3}(t).
\end{align*}

In the end one has

\begin{align*}
t^{\frac 1 4} \|b(t)\|_{ H^r} + t^{\frac 3 4 } \|\Omega(t)\|_{ H^r}  \lesssim  \mathcal{E}_0 + \mathcal{E}_0^{\frac 2 3}\mathcal{M}_r^{\frac 4 3 }(t)+\mathcal{E}_0^{\frac 2 5}\mathcal{M}_r^{\frac 8 5 }(t).
\end{align*}

The next step is to consider
\begin{align*}
&\begin{pmatrix}
\|D_\x^{r-1}  \de_x b(t)\|_{L^2}\\
\|D_\x^{r-1} \de_x \Omega(t)\|_{L^2}
\end{pmatrix} \lesssim \begin{pmatrix}
\min\{1, t^{-3/4}\} \|b_0\|_{W^{1+\nu+r,1}}+\min\{1, t^{-5/4}\} \|\Omega_0\|_{W^{1+\nu+r,1}}\\
\min\{1, t^{-5/4}\} \|b_0\|_{W^{1+\nu+r,1}}+\min\{1, t^{-7/4}\} \|\Omega_0\|_{W^{1+\nu+r,1}}
\end{pmatrix} \\
&\qquad +\int_0^t \begin{pmatrix}
(t-\tau)^{-3/4} \left[\| (-\Delta)^{-1/2} \de_y\Omega \cdot \de_x b\|_{W^{1+\nu+r,1}}+\| (-\Delta)^{-1/2} \de_x\Omega \cdot \de_y b\|_{W^{1+\nu+r,1}}\right]\\
(t-\tau)^{-5/4} \left[\| (-\Delta)^{-1/2} \de_y\Omega \cdot \de_x b\|_{W^{1+\nu+r,1}}+\|(-\Delta)^{-1/2} \de_x \Omega \cdot \de_y b\|_{W^{1+\nu+r,1}}\right]\\
\end{pmatrix} \quad (i)\\
&\qquad +\int_0^t \begin{pmatrix}
(t-\tau)^{-5/4} \left[\| (-\Delta)^{-1/2} \de_x\Omega \cdot \de_y \Omega\|_{W^{1+\nu+r,1}} +\|(-\Delta)^{-1/2} \de_y \Omega \cdot \de_x \Omega\|_{W^{1+\nu+r,1}}\right] \\
(t-\tau)^{-7/4} \left[\|(-\Delta)^{-1/2} \de_x \Omega \cdot \de_y \Omega\|_{W^{1+\nu+r,1}} + \| (-\Delta)^{-1/2} \de_y\Omega \cdot \de_x \Omega\|_{W^{1+\nu+r,1}}\right]
\end{pmatrix} \quad (ii)\\
&\qquad + \int_0^t \begin{pmatrix}
(t-\tau)^{-5/4} \|[(-\Delta)^{-1/2},  (-\Delta)^{-1/2} \de_y\Omega] (-\Delta)^{1/2}\de_x\Omega\|_{W^{1+\nu+r,1}} \\
(t-\tau)^{-7/4} \|[(-\Delta)^{-1/2},  (-\Delta)^{-1/2} \de_y\Omega] (-\Delta)^{1/2}\de_x\Omega\|_{W^{1+\nu+r,1}} \\
\end{pmatrix}\quad (iii)\\
&\qquad + \int_0^t \begin{pmatrix}
(t-\tau)^{-5/4} \|[(-\Delta)^{-1/2}, (-\Delta)^{-1/2}  \de_x\Omega] (-\Delta)^{1/2}\de_y\Omega\|_{W^{1+\nu+r,1}} \\
(t-\tau)^{-7/4} \|[(-\Delta)^{-1/2},  (-\Delta)^{-1/2} \de_x\Omega] (-\Delta)^{1/2}\de_y\Omega\|_{W^{1+\nu+r,1}}
\end{pmatrix}\quad (iv).
\end{align*}

Applying the same reasoning as before, 
\begin{align*}
t^{\frac 3 4} \|\de_x b(t)\|_{ H^{r-1}} + t^{\frac 5 4} \|\Omega(t)\|_{ H^{r-1}}  \lesssim  \mathcal{E}_0 + \mathcal{E}_0^{\frac 2 3}\mathcal{M}_r^{\frac 4 3}(t)+\mathcal{E}_0^{\frac 2 5}\mathcal{M}_r^{\frac 8 5}(t)
\end{align*}

In the previous estimate we used again Lemma \ref{lem-gagliardo} with $s=1+\nu+r$, $s_0=r-1$ and $\theta=\frac 13$, so that $s_1=2+\frac 32 \nu + r$ with $r \ge 0$. 

Summing up all the above together, one has that
\begin{align*}
\mathcal{M}_r(t)  \lesssim  \mathcal{E}_0 + \mathcal{E}_0^{\frac 2 3}\mathcal{M}_r^{\frac 4 3}(t)
+\mathcal{E}_0^{\frac 2 5}\mathcal{M}_r^{\frac 8 5}(t),
\end{align*}

which concludes the proof.
\end{proof}


\begin{remark}[On the regularity assumptions of Theorem \ref{thm:decay-nonlinear}]
In the statement of Theorem \ref{thm:decay-nonlinear}, the initial data $(\Omega_0, b_0) \in W^{1+\nu_0+s, 1}(\R^2) \cap \dot H^{-1}(\R^2)$. Notice that, appealing to Lemma \ref{lem-initial-data}, we see that $(\Omega_0, b_0) \in W^{1+\nu_0+s, 1}(\R^2) \cap \dot H^{-1}(\R^2) \subset \dot H^{s+\tilde \nu} (\R^2) \cap \dot H^{-1}(\R^2) \subset H^{s+\tilde \nu} (\R^2) \cap \dot H^{-1}(\R^2)$ for any choice of $\tilde \nu < \nu_0$. This is the reason why, thanks to \eqref{ineq:nu}, we lose the decay informations on 2 derivatives instead of for instance $2+\frac 3 2 \nu$ as formula \eqref{eq:s1} would suggest. More precisely, in Theorem \ref{thm:decay-nonlinear} we assume that $(b_0, \Omega_0) \in W^{1+\nu_0 + s, 1} \cap \dot H^{-1}$, which implies that the unique solution $(b(t), \Omega(t)) \in H^{s+\tilde \nu} \cap \dot H^{-1}$ for any $0 < \tilde \nu < \nu_0$. Now formula \eqref{eq:s1} implies that we need to control our solution $(b(t), \Omega(t))$ in $H^{r+2+\frac 32 \nu}$. By means of \eqref{ineq:nu}, we see that $ H^{r+2+ \nu_0}  \subset H^{r+2+\frac 3 8 \nu_0} \subset H^{r+2+\frac 32 \nu}$, which is still within the maximal regularity of our solution $(b(t), \Omega(t))$.
\end{remark}
\section{Acknowledgement}
This work has been partially supported by the GNAMPA projects 2019-2020 of INdAM.

\newpage
\appendix
\section{}
The Interpolation Lemma is stated below.
\begin{lemma}[\cite{Chemin2011}, Proposition 1.52]\label{lem-gagliardo}
Let $s_0 \le s \le s_1$. Then
\begin{align*}
\|g\|_{ H^s} \lesssim \|g\|_{ H^{s_0}}^\theta \|g\|_{ H^{s_1}}^{1-\theta} \quad \text{with} \quad   s=\theta s_0+(1-\theta)s_1.
\end{align*}
\end{lemma}

We also employed the following bilinear estimate.

\begin{lemma}[\cite{Wan2019}, Lemma 2.1]\label{lem-bilinear-estimate}
Let $s>0$. Then, for $\frac 1 r = \frac 1 p + \frac 1 q$,
\begin{align*}
\|D_\x^s (fg)\|_{L^r(\R^2)} \lesssim \|f\|_{L^p(\R^2)} \|D_\x^s g\|_{L^q(\R^2)} + \|D_\x^s f\|_{L^p(\R^2)} \|g\|_{L^q(\R^2)}.
\end{align*}
\end{lemma}

Finally, we state a technical result.
\begin{lemma}[\cite{BHN2007}, Lemma 5.2]\label{lem-BHN}
For any $0 \le \gamma, \kappa < 1$ and $t \ge 2$, let $\phi:=\min\{ \gamma, \kappa, \gamma+\kappa-1\}$. Then
\begin{align*}
\int_0^t \min\{1, (t-\tau)^{-\gamma}\} \min\{1, \tau^{-\kappa}\} \, d\tau=t^{-\phi}.
\end{align*}
\end{lemma}


\bibliographystyle{abbrv}
\bibliography{biblio}

\end{document}